\documentclass{article}
\RequirePackage[OT1]{fontenc}
\RequirePackage{amsthm,amsmath}
\RequirePackage{natbib}

\numberwithin{equation}{section}
\theoremstyle{plain}
\newtheorem{theorem}{Theorem}
\theoremstyle{definition}
\newtheorem{lemma}{Lemma}

\newtheorem{remark}{Remark}

\newcommand{\bin}{\{0,1\}^{\mathbb{N}}}

\usepackage[left=4.6cm,right=4.6cm]{geometry}
\usepackage{titlesec}
\titleformat{\section}{\small\centering \rm}{\thesection}{1 em}{}
\titleformat{\subsection}{\normalsize \it}{\thesubsection}{1 em}{}
\usepackage{natbib,titlesec,epsfig,graphicx,color,amsmath,subfigure,amssymb,enumerate,multirow}

\begin{document}
\title{\rm {\normalsize{\bf{A NOTE ON DENSITY ESTIMATION FOR BINARY SEQUENCES}}}}
\author{\small{KARTHIK BHARATH} \thanks{Department of Statistics, Ohio State University, email: bharath.4@osu.edu.}}
\date{}
\maketitle
\begin{abstract}
A histogram estimate of the Radon-Nikodym derivative of a probability measure with respect to a dominating measure is developed for binary sequences in $\bin$. A necessary and sufficient condition for the consistency of the estimate in the mean-square sense is given. It is noted that the product topology on $\bin$ and the corresponding dominating product measure pose considerable restrictions on the rate of sampling required for the requisite convergence.   
\end{abstract}
\section{INTRODUCTION}
We investigate the estimation of the notion of a density, when it exists with respect to a suitable dominating measure, of sequences taking values in $\{0,1\}$ based on an independent sample. Analogous to the intuitive first step towards estimating densities on Euclidean spaces, we construct a histogram estimate and examine its consistency in the mean-square sense. An important aspect of the approach is in the detailed development of several of the ingredients employed in the construction of a histogram on the space $\bin$: the space of infinite binary sequences. In this regard, some interesting challenges surface pertaining to the use of the product topology and product measures on $\bin$; the definition of a suitable Radon-Nikodym derivative and the rate of convergence of the histogram estimator are inextricably linked to the choice of the topology on $\bin$ and the corresponding Borel sets. As it turns out, despite $\bin$ begin a Polish space (with respect to a chosen metric), the usual sufficient condition for consistency of the histogram of the form, $N$ times the volume (or diameter) of the element of a partition of $\bin$ containing the observation of interest converging to infinity, does not work on $\bin$---here $N$ is the size of the sample. The use of a product measure as the dominating measure brings about some complications for the requisite convergence. Our task then is two-fold: construct an appropriate topology on $\bin$ which would generate the type of Borel sets which would engender the definition of a Radon-Nikodym derivative of a probability measure with respect to a chosen dominating measure, meaningful; then, based on the chosen dominating measure, define a suitable histogram estimate which would converge with respect to the mean-square criterion to the true density. 

Several statistical applications involve observations which are binary sequences or are concerned with observations which are coded as binary sequences. However, there appears to be a paucity of work on estimating densities of binary vectors with differing lengths; this is a fairly common scenario encountered in various applications. While it is indeed the case that in practice one typically observes only finite binary sequences, we focus our attention on the more primitive notion of infinite binary sequences. We feel that this is a step in the right direction when dealing with finite binary sequences of differing lengths and are arbitrarily long. To that end, the appropriate space for the binary sequences in question is the Cantor space $\{0,1\}^{\mathbb{N}}$. Properties of the Cantor space have been extensively investigated in work pertaining to algorithmic randomness and dynamical systems---the relevant literature forms an exhaustive list. Fortunately, the books by \cite{DH} and \cite{PS}, and the Ph.D. dissertation of \cite{JR} offer a comprehensive overview of the subjects and contain the necessary bibliographical references;  we direct the interested reader to their works. While, to our knowledge, the problem of estimating densities on $\bin$ has not been explored hitherto, problems of density estimation in separable metric spaces and histogram estimation of Radon-Nikodym derivatives have received attention; see \cite{SDN} and \cite{Oli} in this regard. Indeed, the space$\bin$ is homeomorphic to the Cantor middle thirds set and as will be discussed in the sequel, is a Polish space. However, the unusual structure of the space (it is 0 dimensional with no isolated points) poses additional complications which pose a challenge in employing results from \cite{SDN} and \cite{Oli}. In the subsequent sections, we develop the requisite topology on the space $\bin$ and define a histogram estimator for an appropriate notion of a density, when it exists; we then examine its asymptotic properties. In the final section, we comment on some potential applications of the proposed estimator, highlight its salient features and note some of its shortcomings. 

\section{SETUP AND SOME ISSUES}
Suppose $x_i=(x_{i1},x_{i2},\ldots)$, $i=1,\ldots,N$, are independent sequences taking values in $\bin$ defined on a probability space with probability measure $\mu$. The goal of this article is to, using $x_i$, construct a consistent (in the mean-square sense) histogram estimate of the Radon-Nikodym derivative $\frac{d\mu}{d\lambda}$, whenever it exists, where $\lambda$ is a positive $\sigma$-finite Borel dominating measure and $\mu$ is a positive finite Borel measure absolutely continuous with respect to $\lambda$. We assume throughout that $\mu$ and $\lambda$ are continuous measures; this, as will be elucidated in the sequel, is not just a vacuous proviso but is of considerable importance while showing consistency of the histogram estimate. 

It is well known that elements of $\bin$ can be studied as binary expansions of real numbers from $[0,1]$ via the mapping 
$$\tau \mapsto \displaystyle \sum_{i=1}^\infty \frac{\tau(i)}{2^{i}},$$
where $\tau \in \bin$ and $\tau(i)$ denotes the $i^{\text{th}}$ element of the sequence $\tau$; the map is surjective but not injective. In relation to the expansion, it is also known (see Example 31.1 p.407 of \cite{bill}) that a binary sequence when considered as a manifestation of i.i.d. Bernoulli random variables $\tau(i)$, produces a distribution function which is singular with respect to the Lebesgue measure. If one considers the infinite product measure determined by coin tossing with success probability 1/2, then the random variable $\tau$ has density 1 with respect to the Lebesgue measure of the dyadic interval (length) for all values in $\bin$; if the success probability is different from $1/2$, then the density is 0 rendering singularity. This poses a considerable obstacle in our efforts towards defining the Radon-Nikodym derivative---we are coerced to choose an alternative dominating measure. Fortunately, one is able to consider measures which are topologically equivalent to the coin-tossing measure on $\bin$, i.e. measures which are homeomorphic to the Lebesgue measure on $\bin$. To this end, we state a result (Theorem 3.67) from \cite{nishiura} which shows the existence of a family of bounded Radon-Nikodym derivatives on $\bin$ manufactured via suitable homeomorphisms. 
\begin{theorem}\label{BRD}
Suppose that $\mu$ is a positive, continuous, finite Borel measure on $\bin$ and suppose that $\nu$ is a finite, continuous Borel measure on $\bin$. Then there is a positive real number $c$ and there is a homeomorphism $h: \bin \to \bin$ such that $\nu\ll h \mu$ and $0 \leq \frac{d\nu}{d(h \mu)} \leq c$.
\end{theorem}
Topologically equivalent measures on $\bin$ via the determination of the suitable homeomorphism $h$, has been an area of considerable interest. In this regard, for examples of suitable candidates for $h$, we refer to a couple of articles here from an exhaustive list: \cite{BD1} and \cite{BD2}. We do not delve any deeper regarding singularity issues and existence of bounded Radon-Nikodym derivatives on $\bin$; based on Theorem \ref{BRD}, we assume its existence and wish to estimate it via a histogram. 

Another issue is regarding the choice of the dominating measure. It is worthwhile to note that, for measure-theoretic purposes, the space $\bin$ is isomorphic to the unit-interval and once can therefore equip it with a suitable Lebesgue measure---this measure is usually referred to as a coin-tossing measure. If one constructs the product topology on $\bin$, then there is a unique measure (upto homeomorphism) associated with the measurable space. For instance, if we consider the Borel sets generated by $\{x: x  \text{  is the extension of  } \tau \}$ for some finite sequence $\tau$ of length $k$ and $x \in \bin$, then the Lebesgue measure of such a set would be $2^{-k}$. Indeed, the Lebesgue measure represents a special case of a general product measure. This poses an issue for convergence since it is usually the case that the partition for the construction of the histogram in sample-size based which is of linear order. We are hence left with no choice but to impose restrictions on the sample size and its rate of increase. For reasons offered by the preceding discussion and for the construction of the histogram, which will be elucidated shortly, \emph{we assume that the sample size $N$ is equal to $2^k$ for some positive integer $k$}; if there is no confusion, we shall use $N$ throughout to denote the sample size with the understanding that $N=2^k$.
\section{PRELIMINARIES ON $\{0,1\}^{\mathbb{N}}$}
\subsection{Open sets}
The basic open sets or \emph{cylinders} of $\bin$ are denoted as $[\tau]$, where $\tau$ is a finite binary sequence and $[\tau]$ denotes the set of all points of $\bin$ which are extensions of $\tau$. Observe that if length of $\tau_1$ is lesser than the length of $\tau_2$, denoted by $\tau_1 \leq \tau_2$, then $[\tau_2] \subseteq [\tau_1]$; indeed, if $\tau_2$ is an initial segment of a binary sequences in $\bin$, then so will be $\tau_1$. In other words, finite binary sequences or more accurately, finite binary sequences which are initial segments of infinite ones, induce a topology on $\bin$. The cylinders generate the  Borel $\sigma$-algebra $\mathcal{B}\left(\bin\right)$ on $\bin$(see \cite{DH}). 
\subsection{Measures on $\bin$}
Denote by $\{0,1\}^{<\mathbb{N}}$, the space of all finite binary sequences and denote by $\tau0$ the concatenation of the finite binary sequence $\tau$ and $0$; $\tau1$ is defined along similar lines. Measures on $\bin$ are induced by normalized, monotone, countably additive set functions $\rho: \{0,1\}^{<\mathbb{N}} \to [0,1]$ satisfying
\begin{equation}\label{pm}
\rho(\tau)=\rho(\tau0)+\rho(\tau1) ,
\end{equation}
where $\tau \in  \{0,1\}^{<\mathbb{N}}$. Then by setting $\lambda([\tau])=\rho(\tau)$ one obtains a measure $\lambda$ on $\bin$ (see \cite{DH} for details). The Lebesgue measure on $\bin$ (see, for instance, p.5 of \cite{JR}) assigns to every cylinder its `geometrical size'. Analogous to the case of the unit interval, the Lebesgue measure of any cylinder $[\tau]$ is $2^{-|\tau|}$, where  $\tau$ is the finite binary sequence which generates the cylinder $[\tau]$ and $|\tau|$ denotes the length of $\tau$. Clearly, the Lebesgue measure satisfies the requirement in (\ref{pm}). Therefore, condition (\ref{pm}) places a restriction on the nature of the dominating measure $\lambda$ that can be used in the definition of the histogram estimate. One can consider, as a candidate for $\lambda$, the generalized Bernoulli measure by considering a sequence $p_1,p_2,\ldots,$ such that $0 \leq p_i \leq 1$ for all $i$ and noting that this sequence induces a measure on $\bin$ in the following manner: for each $i \geq 1$, let $\lambda_i(1)=p_i$ and $\lambda_i(0)=1-p_i$ and
\begin{equation}\label{GB}
\lambda([\tau] )=\prod _{i=1}^n \lambda_i(\tau(i)),
\end{equation}
where $\tau$ is a binary sequence of length $n$. If $p_i=1/2$ for each $i$ then the resulting measure is the usual Lebesgue measure. It is easy to verify that the generalized Bernoulli measure satisfies condition (\ref{pm}).

Another choice for $\lambda$ which describes the geometric size of a cylinder in $\bin$ is the Hausdorff measure on $\bin$ (see \cite{JR}). The Hausdorff measure on $\bin$ is an outer measure which generalizes the Lebesgue measure. We do not discuss the details regarding Hausdorff measures on the Cantor space and direct the reader to \cite{JR} for a comprehensive examination of its properties. 
\subsection{Metrics and Partition}
Two metrics which are compatible with the cylinders defined are
\begin{equation}\label{metric1}
d_1(\tau_1,\tau_2)=\frac{1}{N(\tau_1,\tau_2)} 
\end{equation}
and 
\begin{equation}\label{metric2}
d_2(\tau_1,\tau_2)=\frac{1}{2^{N(\tau_1,\tau_2)}} ,
\end{equation}
where $ N(\tau_1,\tau_2)=\text{min}\{k: \tau_1(k) \neq \tau_2(k)\}$.
Under the two metrics and the topology defined above, $\bin$ is a compact metric space and hence separable.  
Using either $d_1$ or $d_2$ we can define the diameter of a set $X \subseteq \bin$ by
$$ d(X)=\sup\{d_i(A,B): A, B \in X \} \quad i=1,2.$$
It is easy to see that under $d_2$, $d([\tau])=\frac{1}{2^{|\tau|}}$ and under $d_1$, $d([\tau])=\frac{1}{|\tau|}$ where $\tau$ is a finite binary sequence which induces the cylinder $[\tau]$. The Lebesgue measure of any cylinder of $\bin$ is equal to its diameter $d([\tau])$ under the metric $d_2$. Indeed, one might be tempted to retain the diameter of the cylinder, under suitable metrics, as the measure itself. But the pitfall with this approach, as exemplified by the metric $d_1$, is that the measure would then fail to satisfy the condition in (\ref{pm}).

The construction of a suitable partition is an important aspect of a histogram estimator. We first note the imperative characteristics of any partition used in the definition of a histogram estimator when consistency is of chief interest. For a collection $\mathcal{C}$ of non-empty subsets of $\bin$, define the mesh of $\mathcal{C}$ by
\[ \text{mesh}(\mathcal{C})=\sup_{X \in \mathcal{C}}d(X)\]
where the diameter $d$ can be defined with respect to either $d_1$ or $d_2$. Along the lines of the approach in \cite{Oli} we consider a sequence of partitions $\Pi_j$, $j \in \mathbb{N}$, of $\bin$ satisfying the following two conditions:
\begin{description}
\item [$A1$.] for each $j \in \mathbb{N}$, the sets in $\Pi_j$ are Borel measurable;
\item [$A2$.] $\text{mesh}(\Pi_j) \to 0$ with increasing sample size.
\end{description}
Now, keeping in mind that the sample size is $N=2^k$, the following Lemma elucidates the construction of a partition on $\bin$ satisfying $A1$ and $A2$.
\begin{lemma}\label{partition}
Consider the sequence of classes of subsets of $\bin$ defined, for each $k \in \mathbb{N}$, by
\[
\Pi_k=\{[\tau]: \tau \in \{0,1\}^k\},
\]
where $N=2^k$. Under the metrics $d_1$ and $d_2$, the sequence $\Pi_k$, is a sequence of partitions of $\bin$ and satisfies conditions $A1$ and $A2$. 
\end{lemma}
The proof is straightforward upon noting that $N \to \infty$ if and only if $k \to \infty$. 
\subsection{Radon-Nikodym derivative and Continuity}
Continuity of a function $g$ at a point $x \in \bin$ can be defined in the following fashion: A function $g:\bin \to (0,\infty)$  is continuous at $x \in \bin$ if for every $\epsilon>0$, there exists an $s \in \mathbb{N}$ such that if $x$ and $x'$  share an initial segment of length $s$, then $|g(x)-g(x')|<\epsilon$. We now make precise the notion of a Radon-Nikodym derivative of the probability measure $\mu$ on $\bin$ with respect to a dominating measure $\lambda$. Denote by $x \upharpoonright_m$, the initial segment of length $m$ of $x \in \bin$. Now, define
\begin{equation}\label{RN}
f(x)=\lim_{m \to \infty}\frac{\mu\left( [x \upharpoonright_m]\right)}{\lambda\left( [x \upharpoonright_m]\right)},
\end{equation}
when the limit exists and is finite. We refer to $f$ as the density of the probability measure $\mu$ with respect to $\lambda$. We wish to estimate $f$ on the basis of $x_i$, $i=1,\ldots,N$ by a histogram estimate. 

For convergence in mean-square, it is usually required to show the convergence of the bias term to zero. In anticipation of such a scenario, we prove a Lemma which is in similar spirit to a density theorem on $\bin$ which characterizes the binary sequences in $\bin$ on which the convergence holds and can be thought of as an analogue of Lebesgue density points; the Lemma follows along the lines of Lemma 2 in \cite{SDN}.
\begin{lemma}\label{LDT}
Consider the sequence of partitions $\{\Pi_k\}$ of $\bin$ defined in Lemma \ref{partition}; denote the $j^{\text{th}}$ element of $\Pi_k$ by $[\tau]_{k,j}$ where $\tau$ is a finite binary sequence of length $k$. Let $g$ be an integrable function with respect to the measure $\lambda$. If $f$ is continuous at $x$ in $[\tau]_{k,j}$ and if $\text{mesh}(\Pi_k) \to 0$, then as $k \to \infty$,
\[
\frac{1}{\lambda([\tau]_{k,j})}\displaystyle \int _{[\tau]_{k,j}} f(t)d\lambda(t) \to f(x).
\]
\begin{proof}
Observe that
\[
\frac{1}{\lambda([\tau]_{k,j})}\left|\displaystyle \int _{[\tau]_{k,j}} f(t)d\lambda(t) -f(x)\right| \leq 
\frac{1}{\lambda([\tau]_{k,j})}\displaystyle \int _{[\tau]_{k,j}} |f(t)-f(x)|d\lambda(t). 
\]
Since $f$ is assumed to be continuous at $x$, for every $\epsilon >0$, there exists an $s \in \mathbb{N}$, such that
$|f(t)-f(x)|<\epsilon$ whenever $t$ and $x$ share a common initial segment of length $s$. Note then that $\delta_{\epsilon}:=d_2(t,x)=2^{-s}$ (the metric $d_1$ could have been used too; there is no qualitative difference). Now, since $\text{mesh}(\Pi_k)\to 0$, for all $\beta>0$, there exists $k_{\beta}$ such that for all $k \geq k_{\beta}$, $\text{mesh}(\Pi_k)<\beta$. Set $\beta=\delta_{\epsilon}$ and we then have that for all $\epsilon>0$ there exists $k_{\epsilon}$ such that for $k \geq k_{\epsilon}$ and $t \in [\tau]_{k,j}$, we have $d_2(t,x)\leq \text{mesh}(\Pi_k)<\delta_{\epsilon}$. This leads us to claim that $|f(t)-f(x)|< \epsilon$. Consequently, for all $\epsilon>0$, there exists $k_{\epsilon}$ such that for $k \geq k_{\epsilon}$
\[
\frac{1}{\lambda([\tau]_{k,j})}\left|\displaystyle \int _{[\tau]_{k,j}} f(t)d\lambda(t) -f(x)\right| \leq\frac{1}{\lambda([\tau]_{k,j})}\displaystyle \int_{[\tau]_{k,j}}\epsilon d \lambda(t)=\epsilon.
\]
This concludes the proof. 
\end{proof}
\end{lemma}

\section{HISTOGRAM ESTIMATE AND MAIN RESULT}
We now have the ingredients required to define the histogram estimator of $f$, the Radon-Nikodym derivative of the probability measure $\mu$ with respect to a dominating measure $\lambda$. To formalize ideas, consider the $\sigma$-algebra of Borel sets of $\bin$ to be $\mathcal{B}\left(\bin\right)$ made up of the cylinders of the form $[\tau]$ where $\tau$ is a finite binary sequence. Assume that $\lambda$ is a positive, Borel, $\sigma$-finite measure on $\left(\bin,\mathcal{B}\left(\bin\right)\right)$ such that $0<\lambda(X)<\infty$ for every cylinder $X$. 

Recall that $x_1,\ldots,x_N$ with $N=2^k$ for some positive integer $k$, are independent binary sequences from $\mu$. Consider the sequence of partitions of the compact metric space $\bin$ given by $\Pi_{k}=\{[\tau]_{k,j}\}$ for $j \in \mathbb{N}$. A histogram estimate of a point $x$ on the real line, roughly speaking, is defined as the proportion of points from the sample falling inside an interval of the partition containing $x$ divided by the length of the interval. The trick then is to get finer and finer partitions with increasing sample size. In an analogous fashion, we first define the estimate of distribution function induced by $\mu$ on Borel sets, based on the sample as
\[
F_N([\tau]_{k,j})=\frac{1}{N}\displaystyle \sum_{i=1}^N \mathbb{I}_{x_i \in [\tau]_{k,j}},
\]
where $\mathbb{I}_A$ is the indicator function of the set $A$. A useful interpretation is as follows: the usual empirical cdf corresponding to a sample of size, say $n$, corresponds to a measure assigning equal mass $n^{-1}$ to each observation; this typifies, in a sense, the ignorance regarding the underlying distribution function. Now, in similar spirit, if one is asked to assign mass to sets generated by a finite binary sequence of length $k$, then the appropriate answer, if no further information is accorded, is $2^{-k}$ corresponding to $k$ tosses of a fair coin. This argument provides the rationale for choosing the sample size $N$ to be equal to $2^k$ for some positive integer $k$. 

The {\it histogram estimator} of $f$ defined in (\ref{RN}), when it exists, is given by
\begin{equation}\label{histogram}
f_N(x)=\frac{F_N([\tau]_{k,j})}{\lambda([\tau]_{k,j})} \quad x \in [\tau]_{k,j}, j \in \mathbb{N}\text{ and } N=2^k.
\end{equation}
This estimator integrates to 1 and is based on the ones examined in \cite{SDN}, \cite{Deh} amongst others. 
\begin{remark}
Note that if $f$ is continuous at $x \in [\tau]_{k,j}$, where $[\tau]_{k,j}$ is as defined in definition of the histogram, then 
\begin{align*}
E(f_N(x))&=\frac{1}{\lambda([\tau]_{k,j})}\int_{[\tau]_{k,j}}f(t)d\lambda(t)\\
&\to f(x) 
\end{align*}
as $N \to \infty$ from Lemma \ref{LDT}.
\end{remark}
We now prove that $f_N$ defined in (\ref{histogram}) is a consistent estimator of $f$ with respect to the mean-square criterion. Denote by $E$ and $V$, the expectation and the variance, respectively, taken with respect to the density $f$. 
\begin{theorem}\label{MSE}
Let $N=2^k$ for some positive integer $k$. Consider the sequence of partitions of $\bin$, $\{\Pi_k=\{[\tau]_{k,j}: \tau \in \{0,1\}^k, j \in \mathbb{N}\}\}$ with $\text{mesh}(\Pi_k) \to 0$. Then, for all $x \in [\tau]_{k,j}$ if $f$ is continuous at $x$, then as $N \to \infty$, 
\[
E\left([f_N(x)-f(x)]^2\right) \to 0
\]
if and only if $N \lambda([\tau]_{k,j}) \to \infty$.
\end{theorem}
\begin{proof}
We shall first prove necessity. First note that
\[
E\left([f_N(x)-f(x)]^2\right)=E^2\left(f_N(x)-f(x)\right)+V(f_N(x)).
\]
If $E\left([f_N(x)-f(x)]^2\right) \to 0$ then it is necessarily true that both bias and variance converge to $0$. In particular, upon examination of the variance term, we obtain
\begin{align*}
V(f_N(x))&=E(f^2_N(x))-E^2(f_N(x))\\
&=E\left[\left(\frac{F_N([\tau]_{k,j})}{\lambda([\tau]_{k,j})}\right)^2\right]-E^2\left[\frac{F_N([\tau]_{k,j})}{\lambda([\tau]_{k,j})}\right]\\
&=\frac{1}{N \lambda^2([\tau]_{k,j})}\int_{[\tau]_{k,j}}f(t)d\lambda(t)-\frac{1}{\lambda^2([\tau]_{k,j})}\left[\int_{[\tau]_{k,j}}f(t)d\lambda(t)\right]^2\\
&\leq\frac{E(f_N(x))}{N\lambda([\tau]_{k,j})}.
\end{align*}
Since the variance, $V(f_N(x)) \to 0$ for $x \in [\tau]_{k,j}$ and by Lemma \ref{LDT}, $E(f_N(x))\to f(x)$, we must have $N \lambda([\tau]_{k,j}) \to \infty$. We now prove sufficiency. Consider the expression for the variance:
\begin{align*}
V(f_N(x))&=\frac{E(f_N(x))}{N\lambda([\tau]_{k,j})}-\frac{E^2(f_N(x))}{N}\\
&\leq \frac{E(f_N(x))}{N\lambda([\tau]_{k,j})}.
\end{align*}
Now, if $\text{mesh}(\Pi_k) \to 0$ and $N \lambda([\tau]_{k,j}) \to \infty$, then $V(f_N(x)) \to 0$. We now turn our attention to the bias term and show that it converges to 0 as $N \to \infty$; indeed, then its square would too by a continuous mapping argument.  
\begin{align*}
E\left(f_N(x)-f(x)\right)&=E(f_N(x))-f(x))\\
&=\frac{1}{ \lambda([\tau]_{k,j})}\int_{[\tau]_{k,j}}f(t)d\lambda(t)-f(x)\\
& \to 0
\end{align*}
from Lemma \ref{LDT} since $\text{mesh}(\Pi_{k})\to 0$ as $N \to \infty$ . This concludes the proof. 
\end{proof}
Theorem \ref{MSE} informs us that one needs to increase sample size at an exponential rate in order to obtain a consistent histogram estimate. This represents a highly restrictive condition and is a direct artifact of the dominating product measure which directs the definition of the histogram estimate. Observe that as a consequence of Theorem \ref{MSE}, we can claim that $f_N(x)$ converges in probability to $f(x)$ for all $x \in \bin$ at which $f$ is continuous. It is also true that, owing to Lemma \ref{LDT},  $E(f_N(x))$ converges to $f(x)$ for $x$ at which $f$ is continuous. At this juncture, it is natural to query the rate at which these two convergences occur relative to one another. The following Theorem sheds light on the issue when $f$ is bounded. The proof follows along the lines of the proof of Lemma 2.3.2 in \cite{BLS}.
\begin{theorem}
Under the conditions of Theorem \ref{MSE} suppose, additionally, that $f$ is bounded on $\bin$. Then, 
\[
E\left[\left(f_N(x)-E(f_N(x))\right)^{2m}\right]=O\left(\left(N\lambda([\tau]_{k,j}) \right)^{-m}\right), \quad x \in [\tau]_{k,j},
\]
for any integer $m \geq 1$. 
\end{theorem} 
\begin{proof}
Note that
\[
f_N(x)-E(f_N(x))=\frac{1}{N\lambda([\tau]_{k,j})} \left[\alpha_1+\cdots+\alpha_N\right],
\]
where $$\alpha_i=\mathbb{I}_{x_i \in [\tau]_{k,j}}-\int_{[\tau]_{k,j}}f(t)d\lambda(t), \quad i=1,\ldots,N.$$
For any integer $s \geq 1$, 
\begin{align*}
E(\alpha_i^s) &\leq E\left[\left(\mathbb{I}_{x_i \in [\tau]_{k,j}}\right)^s\right]\\
&=E\left[\mathbb{I}_{x_i \in [\tau]_{k,j}}\right]\\
&=\int_{[\tau]_{k,j}}f(t)d\lambda(t)\\
&\leq K \lambda([\tau]_{k,j})
\end{align*}
where $K$ is the bound on $f$. This implies that 
\[
E(\alpha_1+\cdots+\alpha_N)=O\left( \lambda([\tau]_{k,j})\right)
\]
and the main contribution to $E\left[\left(\alpha_1+\cdots+\alpha_N\right)^{2m}\right]$ is from the cross-product terms. Therefore,
\begin{align*}
E\left[\left(f_N(x)-E(f_N(x))\right)^{2m}\right]&=\left(N\lambda([\tau]_{k,j})\right)^{-2m}E\left[\left(\alpha_1+\cdots+\alpha_N\right)^{2m}\right]\\
&=O\left(\left(N\lambda([\tau]_{k,j}) \right)^{-m}\right).
\end{align*}
\end{proof}
\section{DISCUSSION}
Assumptions of independence or perhaps Markov dependence are usually imposed on the elements of the observed binary sequences before estimating densities. What is noteworthy in our setup is the absence of any such structural assumption---we do not assume that the infinite coin tosses which generate a sequence are independent, exhibit some sort of Markov dependence or satisfy any type of mixing condition. The histogram is, in that regard, an apt estimator. Several modifications are necessary for the application of the framework and the results provided in this article to statistical problems involving sequences of binary observations; the obvious one is regarding the compatibility of finiteness of the sequences observed in practice to nature of the theory presented here regarding infinite binary sequences. However, it is to be noted that the partition $\Pi_k$ of $\bin$ where the sample size $N=2^k$ is constructed from binary sequences of length $k$. For example, if $k=3$, then the partition $\Pi_3$ of $\bin$ is
\[
\Pi_3=\bigg\{[000],[001],[010],[011],[100],[101],[110],[111]\bigg\},
\]
where, for example, $[000]$ is the set of all binary sequences which extend $000$. Indeed, any arbitrary finite binary sequence $\tau$ will belong to one of the sets in $\Pi_3$ and an estimate of the density at $\tau$ can, in principle, be obtained upon fixing a dominating measure. An important area of potential application of the proposed estimator is in the study of algorithmic randomness. Null sets of $\bin$ play a crucial role in this field as they form the basis for tests which determine if a binary sequence is random---see Martin L\"{o}f tests in \cite{DH} for more details. It is known that null sets of $\bin$ are points where a measure absolutely continuous with respect to the Lebesgue measure on $\bin$ has infinite density. It is conceivable that our estimator, based on $N$ realizations of the random binary sequence in question, can be used to check for null sets. 

A modern interesting application is one concerning tree-structured data. In a recent article, \cite{balding} formulated a technique by which data in the form of rooted trees were coded as binary sequences via a bijective mapping; they then defined a probability space from which a random sample of such trees were drawn from and proceeded to prove limit theorems. The interesting aspect of their construction was the fact that each node of the tree could have only have a finite number of children but the tree could, in principle, have unbounded number of generations, viz. infinite number of nodes. \cite{busch}, using the construction and results from \cite{balding}, developed a goodness-of-fit test for tree-structured data arising from a protein classification problem. Indeed, the trees in the dataset considered consisted of finitely many nodes and in order to use results from \cite{balding}, they assumed that the probability measure of interest assigned positive mass to only trees with finite number of nodes. Another interesting application concerns phylogenetics: while modeling phylogenetic trees using Bayesian techniques one typically draws an MCMC sample of trees from a posterior distribution; several summary statistics are considered at this point amongst which the consensus tree is popular. It would be interesting to investigate if the construction in \cite{balding} can be employed in conjunction with our estimator to develop statistical procedures and summary statistics on the space of binary sequences via bijective mappings from the tree space. More recently, \cite{RSLW} considered the problem of estimating the density of $d$-dimensional Bernoulli vectors and developed an estimator which attains near-minimax MSE in the case where $N <2^d$. Indeed, their work is primarily concerned with the scenario where all the observed binary sequences are of the same length $d$.  

Another area where lengthy binary sequences play a prominent role is genomic studies wherein a binary sequence $\tau_i$ could provide a biochemical profile of a bacterial strain where every component is a ``yes-no" indicator of a particular biochemical marker. If the lengths of $\tau_i$ are all the same, then existing methods for density estimation can be used. However, it is not unreasonable to consider the situation when the $\tau_i$ are of differing lengths corresponding to differently measures profiles.  Applications in social sciences are rife with binary sequences corresponding to ``yes-no" questions in questionnaires, record of event or non-events etc; likewise, data in the form of binary sequences is common in several computer science applications as answers to database queries etc. Despite recent advances in functional data techniques assuming observed trajectories of functions, questions regarding functional data in the form of observed binary process have received scant attention. There have been a few attempts based on viewing the observed binary sequence from a parametric perspective assuming exponential family of distributions generated via a link function (see \cite{AVD}). The choice of the link function, as one might imagine, would play a critical role in any ensuing methodology.  In that regard, the framework presented here can be viewed as a first step towards a non-parametric setting, but perhaps exploratory in spirit. In all these applications, when the observed sequences are of varying lengths, one is unsure concerning the circumscription of the sequences to equal lengths based on a scientific reason; any choice would be arbitrary and would severely impact subsequent inference. Noting the construction of the partition based on finite sequences, our estimator may well come in handy in such settings. It would be interesting to examine the utility of the proposed estimator under such settings; much work remains to be done in this direction. 

From a theoretical perspective, however, the Cantor space has been a rich source of problems in theory of algorithmic randomness, geometric measure theory etc. But its use as a candidate modeling space in statistical literature has been minimal. The conduciveness of the space in defining various family of probability measures, in principle, renders it attractive as a modeling space for binary data; we feel more investigation in this direction might prove to be fruitful. One possible direction for extending the results presented here lies in the exploitation of the fractal nature of the Cantor space via Hausdorff measures as the dominating measures. Additionally, connections with topological entropy of the binary sequences and Hausdorff dimensions established for the Cantor space in \cite{bill2} are also deserving of further consideration. 
\section{ACKNOWLEDGEMENTS}
The author is grateful to Prem Goel and Sebastien Kurtek for helpful discussions and particularly thanks Steve Maceachern for suggesting this problem in the first place. He also thanks a referee for a careful reading of the paper and useful suggestions which improved the overall exposition.

\bibliography{ref_cantor}
\bibliographystyle{plainnat}

\end{document}